\documentclass[11pt]{article}
\usepackage{amsmath, graphicx, amsfonts,amssymb}
\usepackage{amsfonts}
\usepackage{mathrsfs}
\usepackage{amsthm}

\usepackage[pdfstartview=FitH,
            CJKbookmarks=true,
            bookmarksnumbered=true,
            bookmarksopen=true,
            colorlinks=true,
            linkcolor=blue,
            anchorcolor=blue,
            citecolor=red,
            urlcolor=blue
            ]{hyperref}

\numberwithin{equation}{section}
\newtheorem{thm}{Theorem}[section]

\newtheorem{lem}[thm]{Lemma}
\newtheorem{cor}[thm]{Corollary}

\theoremstyle{remark}

\theoremstyle{definition}
\newtheorem{definition}{Definition}
\newtheorem*{Def}{Definition}

\def\bl{\begin{lemma}}
\def\el{\end{lemma}}
\def\bc{\begin{corollary}}
\def\ec{\end{corollary}}
\def\bt{\begin{theorem}}
\def\et{\end{theorem}}

\def\bp{\begin{proposition}}
\def\ep{\end{proposition}}
\def\be{\begin{equation}}
\def\ee{\end{equation}}

\def\bd{\begin{definition}}
\def\ed{\end{definition}}

\newcommand{\wt}{\widetilde}

\newcommand{\tr}{\mathbb{R}}

\newcommand{\rn}{\mathbb{R}^n}
\newcommand{\kn}{\mathcal{K}^n}
\newcommand{\sn}{S^{n-1}}

\newcommand{\tk}{\mathcal{K}}

\newcommand{\lh}{\left}
\newcommand{\rh}{\right}
 \newcommand{\vol}{\operatorname{vol}}

\title{The Reverse-log-Brunn-Minkowski inequality
 }

\author{Dongmeng Xi}

\begin{document}

\date{}
\maketitle

\begin{abstract} Firstly, we propose our conjectured Reverse-log-Brunn-Minkowski inequality (RLBM).

	Secondly, we show that the (RLBM) conjecture is equivalent to the log-Brunn-Minkowski (LBM) conjecture proposed by B\"or\"oczky-Lutwak-Yang-Zhang. We name this as  ``reverse-to-forward" (RTF) principle.
		Using this principle, we give a very simple new proof of the log-Brunn-Minkowski inequality in dimension two.
	
	Finally, we establish the (RTF) principle for the conjectured log-Minkowski inequality (LM).
	Using this principle, we prove the
	log-Minkowski inequality in the case that one convex body is a zonoid (the inequality part was first proved by van Handel).
Via a study of the lemma of relations,
the full equality conditions (the bodies have ``dilated direct summands") are also characterized, which turns to be new.
\end{abstract}

\section{Introduction}

Let $n\ge 2$ be an integer.
In their seminal work \cite{BLYZ12,BLYZ13}, B\"or\"oczky-Lutwak-Yang-Zhang (B\"or\"oczky-LYZ) proposed the following
\vskip 4pt
\noindent{\bf Log-Brunn-Minkowski (LBM) conjecture}: For all origin-symmetric convex bodies  $K_0$ and $K_1$, it holds that
\be\label{lbmin} V(K_t) \ge V(  K_0)^{1-t}V(  K_1)^t, \quad \forall t\in (0,1),\ee
where 
\begin{equation}\label{path1}
 K_t = \{ x\in\rn: x\cdot v  \le h_{K_0}(v)^{1-t}h_{K_1}(v)^{t}, \forall v\in\sn  \}.     
\end{equation}
Equality is achieved if and only if
$K_0$ and $K_1$ have dilated direct summands. 
\vskip 4pt
They verified the conjecture in dimension 2. 
They also proved that (LBM) conjecture   is equivalent to the following
\vskip 4pt
\noindent{\bf Log-Minkowski (LM) conjecture.} For all origin-symmetric convex bodies  $K_0$ and $K_1$, it holds that
\be\label{lmi} \int_{\sn} \log \frac{h_{K_1}}{h_{K_0}} dV_{K_0} \ge \frac{V(K_0)}n \log\frac{V(K_1)}{V(K_0)},\ee
with equality if and only if
$K_0$ and $K_1$ have dilated direct summands. Here $V_K$ is called the {\it cone-volume measure}, defined for Borel subsets of $\sn$, as the volume of the convex hull of the origin and the pre-image of the Borel subset under the generalized Gauss map $\nu_K$ (where almost all boundary points have unique outer normal vectors),  
\[ V_K(\omega) = V({\rm conv}\{ o, \nu_K^{-1}(\omega)\}), \quad {\rm for~Borel~} \omega \subset \sn. \] 
\vskip 4pt
Their work \cite{BLYZ12,BLYZ13} also suggested that these conjectures are 
 equivalent to  
\vskip 4pt
\noindent{\bf Uniqueness conjecture of the logarithmic Minkowski problem.} If  $K_0$ and $K_1$ are origin-symmetric convex bodies such that
\[V_{K_0} = V_{K_1},\]
then $K_0$ and $K_1$ have dilated direct summands. 

\vskip 4pt
In recent years, a significant breakthrough was obtained by
Kolesnikov-Milman \cite{KM} (local version) and Chen-Huang-Li-Liu \cite{CHLL} (local-to-global),  in which  they proved the $L_p$-Brunn-Minkowski conjecture for $p$ slightly smaller than 1  and  
showed that the (LM) conjectured inequality is equivalent to a spectral-estimate of the so-called `Hilbert-Brunn-Minkowski operator'.  
For the case that the `base body' $K$ is a zonoid,
van Handel \cite{H2} proved the (LM) inequality, by  estimating the  spectral-gap of the Hilbert-Brunn-Minkowski operator $K$. But the conjectured full equality conditions seem to be open.

\vskip 4pt

\noindent{\bf A quick review of the equivalence of the   conjectures by B\"or\"oczky-LYZ.} The (LBM) conjecture says that the function 
$$f(t) = \log V(K_t)$$  
should be concave on $(0,1)$. If (LBM) is affirmative, then by   concavity, 
$$f'(0) \ge f(1) - f(0) \ge f'(1).$$
Meanwhile,  Aleksandrov's variation formula implies
\be \label{alek}
f'(0) =  \int_{\sn} \log \frac{h_{K_1}}{h_{K_0}} dV_{K_0}, \quad f'(1) =  \int_{\sn} \log \frac{h_{K_1}}{h_{K_0}} dV_{K_1}.\ee
Then,  (LM)  follows.  
The equality characterization follows from the strict concavity of $f$ (if $K_0$ and $K_1$ do not have dilated direct summands).  
For the uniqueness problem, if $V_{K_0}=V_{K_1}$, then  
$K_0$ and $K_1$ should have dilated direct summands  because $f'(0)=f'(1)$. 
Here, we omit the review of   the part concerning the `uniqueness conjecture implying (LM) conjecture and (LBM) conjecture'. 
 \vskip 4pt
If we view the `uniqueness conjecture' as the ultimate goal (actually equivalent to (LBM)), then the insight from the above review can be extracted as follows. Along with the `path' $K_t$ given by  \eqref{path1}, one can take differential to obtain \eqref{alek}. And since $K_1$ can be very arbitrary, \eqref{alek} actually describes all the information of the measure $V_{K_0}$. This makes us to consider the following nature question: 

\vskip 4pt

\noindent{\bf Question.} 
Is there the other  `path' $\wt K_t$, such that $t\mapsto V(\wt K_t)$ has a certain kind of convexity, and such that one can differentiate $V(\wt K_t)$ and the differential also describes all the information of the cone-volume measure $V_{K_0}$?

\vskip 4pt

The idea can be illustrated as follows. Studying the log-concavity of the volume functional along the 'old path' (the (LBM) conjecture) seems too difficult to approach directly. However, one might attempt to identify a 'new path' as suggested in the above question, along which it may be easier to investigate a certain convexity of the volume functional. This approach could then lead to achieving the ultimate goal-uniqueness conjecture.

The new path in the current paper is defined as follows. 

\begin{Def}
For  $K_0, K_1 \in \kn_e$,
define $\wt K_t$, for each $t\in (0,1)$, to be the unique origin-symmetric convex body that solves the following Minkowski problem
\be \label{d1} dS_{\wt K_t}(v) = \left[\frac{dS_{K_0}}{d(S_{K_0}+S_{K_1})}(v)\right]^{1-t}\left[\frac{dS_{K_1}}{d(S_{K_0}+S_{K_1})}(v)\right]^{t}d(S_{K_0}+S_{K_1})(v),\ee
where $ {dS_{K_i}}/{d(S_{K_0}+S_{K_1})}$, for $i=0,1$,
is the Radon-Nikodym differential. Here the Radon-Nikodym differential exists
because that $S_{K_i}$ is always absolutely continuous with respect to $d(S_{K_0}+S_{K_1})$. We also denote it by
\be\label{def-logblaschke} \wt K_t = (1-t)\cdot K_0 \#_0 t\cdot K_1, \notag\ee
and name it as {\it log-Blaschke combination}.  

The general expression in \eqref{d1} may seem a little complicated, but there are two other special cases that are useful and easy to understand. If  $S_{K_0}$ and $S_{K_1}$ are absolutely continuous w.r.t. each other, then 
\be \label{d2} dS_{\wt K_t} = \left(\frac{dS_{K_1}}{dS_{K_0}}\right)^t dS_{K_0}\ee
If both of $S_{K_0}$ and $S_{K_1}$ have density functions $\varphi_0$ and $\varphi_1$, then $\wt K_t$ is the convex body whose support function is the unique even solution to the following Monge-Amper\'e equation on $\sn$
\[ \det(\nabla^2h + hI) = \varphi_0^{1-t}   \varphi_1^{t}.\]

\end{Def}
In general, $\wt K_t$ will converge to a body contained in $K_0$ as $t\to 0$, and similarly as $t\to 1$. Under the assumption that $S_{K_0}$ and $S_{K_1}$ are absolutely continuous w.r.t. each other, we can prove (Lemma \ref{lem2.3})
\[\lim_{t\to0}\wt K_t = K_0 \quad {\rm and}\quad  \lim_{t\to1}\wt K_t = K_1.\]
The convergence is in Hausdorff metric.

Proposed here is,
\vskip 4pt

\noindent{\bf Reverse-log-Brunn-Minkowski (RLBM) conjecture.}  Let $K_0, K_1 \in\kn_e$. Then,
\be V(\wt K_t) \le V(  K_0)^{1-t}V(  K_1)^t, \quad \forall t\in (0,1).\ee
Equality is achieved if and only if
$K_0$ and $K_1$ have dilated direct summands.

\vskip 4pt

In the subsequent context, we will proceed with the actions outlined in the Abstract.
In Section 2, we show that the (RLBM) conjecture   is equivalent to the log-Brunn-Minkowski (LBM) conjecture. 
\footnote{The equivalence of the inequalities  was also noticed independently by Emanule Milman by using the ``local-to-global" principle.} 
Indeed, we prove in Theorem 2.6 that the (LBM)  inequality is equivalent to the (RLBM) inequality (without assuming equality characterization were proved), and is equivalent to a ``seemingly weaker" inequality 
\be\label{weakrlbm} V(\wt K_t) \le V(K_t). \ee
We prove in Corollary 2.7 that if one of the inequalities is verified, then the equality characterizations of (LBM) and (RLBM) are also equivalent, and is equivalent to the equality characterization of \eqref{weakrlbm} . We refer to these results as   ``reverse-to-forward" {\bf (RTF)} principle.

In Section 3, as an application of the (RTF) principle, we give a very simple new proof of the log-Brunn-Minkowski inequality in dimension two, by showing  $\wt K_t \subset K_t$. In higher dimension, \eqref{weakrlbm} also provides a new strategy: if one can construct a middle convex body $\bar K_t$ , such that   
\[ V(\bar K_t) \ge V(\wt K_t) \quad {\rm and} \quad \bar K_t\subset K_t, \]
then one can verify the (LBM) conjecture. 

In Section 4, we establish the (RTF) principle for the (LM) inequality.
Based on this principle, we prove the  (LM) inequality in the case that one convex body is a zonoid (the inequality part was first proved by van Handel). The conjectured full equality conditions are also characterized, which turns to be new.

\section{The equivalence of (LBM) and (RLBM)}

We will prove that   (RLBM) conjecture is equivalent to (LBM) conjecture in this section. Before this, we need some preparations.
	
The following is a variation formula with respect to a pertubation of the surface area measure. The idea of its proof is in spirit of Aleksandrov's approach to prove his well-known variation formula.

\begin{lem}[One-sided variation formula for the perturbation of surface area measure]\label{var}
	Let  $K\in \kn_e$. Let	$\varphi\in L^1(\sn;S_K)$ be an  even  function,  such that $\varphi$ is positive $S_K$-a.e. Define $\wt K_{\varphi,t}$ to be the unique origin-symmetric convex body satisfying
	\[dS_{\wt K_{\varphi,t}} = \varphi^t dS_K, \quad t\in [0,1].\]
	Then,
	\[ \frac{d}{dt}\Big|_{t=0^+} V(\wt K_{\varphi,t}) =  \frac1{n-1} \int_{\sn} \log \varphi h_{K}dS_K.\]
Note that the right hand side is bounded from above, but might be $-\infty$.
\end{lem}

\begin{proof}   By elementary calculus, we see that
\[ g(t) =  \frac{\varphi^t-1}t,  \quad t\in \tr\]
is increasing in $t$. In particular,
\be\label{ebd*}   \varphi^t-1 \le t(\varphi-1), \quad \forall t\ge 0.\ee
As $t\to 0^+,$
\[\frac{\varphi^t-1}t \quad {\rm  converges  ~to} \quad \log \varphi \quad {\rm decreasingly}.\]
When $t=1$, the function $\varphi -1 $ is integrable w.r.t the measure $h_KdS_K$. Thus, by the monotone convergence theorem, we have
\be\label{ec2} \int_{\sn} \frac{\varphi^t - 1}t h_KdS_{K}\to \int_{\sn} \log\varphi h_KdS_{K} \quad {\rm decreasingly}, \quad {\rm as}~ t\to 0^+.\ee
Note that, we also have
\[\int_{\sn} \log\varphi h_KdS_{K} \le \int_{\sn} (\varphi - 1) h_KdS_{K}<\infty,\]
but the left hand side might be $-\infty$.

Next, we proceed our proof and split it into three steps.
\vskip 4pt

	\noindent {\it Step 1.} We prove that $S_{\wt K_{\varphi,t}}\to S_K$ weakly, and hence $\wt K_{\varphi,t}\to K$ in Hausdorff metric.

Let $f\in C(\sn)$. 	 By \eqref{ebd*}, we have
\[|f(\varphi^t -1)|  \le \|f\|_\infty \cdot \max\{ 1, \varphi-1\} \in L^1(\sn; S_K).\]
By this and the fact that $\varphi^t\to 1$ a.e. as $t\to 0$, and by the dominated convergence theorem, we obtain
\[\int_{\sn} f(\varphi^t-1) dS_{K}\to 0, \quad{\rm as}~t\to 0,\]
the weak convergence of the surface area measure $S_{\wt K_{\varphi,t}}$.

\vskip 4pt
	\noindent {\it Step 2.} 	
Note that  Step 1 implies that $V(\wt K_{\varphi,t})\to V(K)$. Then,	it is simple to see that
	\be\label{limsups}
\begin{aligned} \frac1n \limsup_{t\to 0^+} \frac{V(\wt K_{\varphi,t}) - V(K)}t =& \frac1{n-1} \limsup_{t\to 0^+}V(\wt K_{\varphi,t})^{1/n}\frac{V(\wt K_{\varphi,t})^{(n-1)/n} - V(K)^{(n-1)/n}}t \\
=& \frac1{n-1} V(K)^{1/n}\limsup_{t\to 0^+} \frac{V(\wt K_{\varphi,t})^{(n-1)/n} - V(K)^{(n-1)/n}}t,  \end{aligned}\ee
and
	\be\label{liminfs}
\begin{aligned} \frac1n \liminf_{t\to 0^+} \frac{V(\wt K_{\varphi,t}) - V(K)}t =& \frac1{n-1} \liminf_{t\to 0^+}V(\wt K_{\varphi,t})^{1/n}\frac{V(\wt K_{\varphi,t})^{(n-1)/n} - V(K)^{(n-1)/n}}t \\
=& \frac1{n-1} V(K)^{1/n}\liminf_{t\to 0^+} \frac{V(\wt K_{\varphi,t})^{(n-1)/n} - V(K)^{(n-1)/n}}t. \end{aligned}\ee

By Minkowski's inequality for the mixed volume, we derive that
\be\label{ebd1}\begin{aligned}
V(K)^{1/n}\cdot \frac{V(\wt K_{\varphi,t})^{(n-1)/n} - V(K)^{(n-1)/n}}t
\le& \frac{V_1(\wt K_{\varphi,t},K) - V(K)}t\\
=&  \frac1n\int_{\sn}  \frac{\varphi^t - 1}t  h_KdS_{K},\end{aligned}\ee
and
\be\label{ebd2}\begin{aligned}
V(K_{\varphi,t})^{1/n}\cdot \frac{V(\wt K_{\varphi,t})^{(n-1)/n} - V(K)^{(n-1)/n}}t
\ge& \frac{V(\wt K_{\varphi,t}) - V_1(K, \wt K_{\varphi,t})}t\\
=&  \frac1n\int_{\sn} \frac{ \varphi^t -1  }t h_{\wt K_{\varphi,t}} dS_{K}.\end{aligned}\ee

By \eqref{ebd1}, \eqref{limsups}, and \eqref{ec2}, we have
\be\label{ine-sup} \limsup_{t\to 0^+} \frac{V(\wt K_{\varphi,t}) - V(K)}t \le \frac1{n-1} \int_{\sn}  \log\varphi    h_KdS_{K}.\ee
If $\int_{\sn} \log\varphi  h_KdS_{K} = -\infty$, then it follows from \eqref{ebd1} and \eqref{limsups} that
\[ \lim_{t\to 0^+} \frac{V(\wt K_{\varphi,t}) - V(K)}t = \limsup_{t\to 0^+} \frac{V(\wt K_{\varphi,t}) - V(K)}t = \int_{\sn}  \log\varphi    h_KdS_{K} = - \infty,\]
and we  complete the proof.

Now we suppose $\int_{\sn}  \log\varphi  h_KdS_{K}$ is finite.
Since $h_K$ is bounded from below, we also have
\[\lh|\int_{\sn}  \log\varphi  dS_{K}\rh| < \infty.\]
Since Step 1 implies that $h_{\wt K_{\varphi,t}}\to h_K$ uniformly,  we have
\[ \int_{\sn} \frac{ \varphi^t -1  }t h_{\wt K_{\varphi,t}} dS_{K} \to \int_{\sn}  \log\varphi  h_KdS_{K}, \quad {\rm as}~ t\to 0.\]
Then, by  \eqref{ebd2}, \eqref{liminfs}, and \eqref{ec2},
we have
\[\liminf_{t\to 0^+} \frac{V(\wt K_{\varphi,t}) - V(K)}t \ge \frac1{n-1}\int_{\sn} \log \varphi h_KdS_K.\]
	As a result, we see that the limit exists, and
	\[ \frac{d}{dt}\Big|_{t=0^+} V(\wt K_{\varphi,t}) =  \frac1{n-1} \int_{\sn} \log \varphi h_{K}dS_K.\]
\end{proof}

\begin{lem}[Two-sided variation formula for the perturbation of surface area measure]\label{var2}
	Let  $K\in \kn_e$. Let	$\varphi$ be an even  function, such that both $\varphi$ and $\varphi^{-1}$ belong to $L^1(\sn;S_K)$, and such that $\varphi$ is positive $S_K$-a.e.  Define $\wt K_{\varphi,t}$ to be the unique origin-symmetric convex body satisfying
	\[dS_{\wt K_{\varphi,t}} = \varphi^t dS_K, \quad t\in [-1,1].\]
	Then,
	\[ \frac{d}{dt}\Big|_{t=0} V(\wt K_{\varphi,t}) =  \frac1{n-1} \int_{\sn} \log \varphi h_{K}dS_K.\]
\end{lem}

\begin{proof}	
It suffices to consider the case $t<0$.
Notice that
\[\varphi^{t} = \lh(\varphi^{-1}\rh)^{-t},\]
with $ -t >0 $. By Lemma \ref{var}, if we denote $s=-t$, and denote $\wt L_s = \wt K_t$, then we obtain
\[ \frac{d}{dt}\Big|_{t=0^-} V(\wt K_{\varphi,t}) = - \frac{d}{dt}\Big|_{s=0^+} V(\wt L_{\varphi,s}) =  \frac1{n-1} \int_{\sn} \log \varphi h_{K}dS_K.\]
Moreover, we know that $\log\varphi \in L^1(\sn;S_K)$ in this case.
\end{proof}

The following lemma will be used.

\begin{lem}\label{smalllem}
	Let $K,L\in \kn_e$.
	If $S_K$ and $S_L$ are absolutely continuous with respect to each other,
	then
	\[\varphi(v) = \frac{dS_L}{dS_K}(v)\]
	is positive and finite $S_K$-a.e. Moreover, $\varphi\in L^1(\sn; S_K)$ and $\varphi^{-1} \in L^1(\sn; S_L)$.
\end{lem}

\begin{proof}
	We first prove that $\varphi = {dS_L}/{dS_K} $ is positive $S_K$-a.e.
	Otherwise, there is $\eta\subset \sn$ with $S_K(\eta)>0$, and
	\[\varphi(v)= 0, \quad \forall v\in\eta.\]
	Then, $S_L(\eta)=0$ but $S_K(\eta)>0$, which contradicts the fact that $S_L$ is  absolutely continuous with respect to $S_K$. Meanwhile, since \[S_L(\sn) = \int_{\sn} \varphi dS_K,\]
it is clear that $\varphi\in L^1(\sn; S_K)$.

	Likewise, $\varphi^{-1} ={dS_K}/{dS_L} $ belongs to $ L^1(\sn; S_:)$, and it	is positive $S_L$-a.e. Since an $S_L$-null set must be an $S_K$-null set, we know that $\varphi$ is finite $S_L$-a.e.
\end{proof}

\begin{lem}\label{lem2.3}
	Let $K,L\in \kn_e$.
	If $S_K$ and $S_L$ are absolutely continuous with respect to each other,
	then
	\[\lim_{t\to0}\wt K_t = K_0, \quad{\rm and} \quad \lim_{t\to1}\wt K_t=K_1.\]
\end{lem}

\begin{proof}
By Lemma \ref{smalllem},
\[\varphi = \frac{dS_{K_1}}{dS_{K_0}}\]
is positive and finite $S_{K_0}$-a.e. Therefore,
$\varphi^t \to 1$ as $t\to 0$, and  $\varphi^t \to \varphi$ as $t\to 1$,  $S_{K_0}$-a.e. Similar to Step 1 of the proof of Lemma \ref{var}, using the dominated convergence theorem, we will complete the proof of this lemma.
\end{proof}

\begin{thm}\label{var-thm} Let $K_0,K_1\in \kn_e$.
	If $S_{K_0}$ and $S_{K_1}$ are absolutely continuous with respect to each other,
	then $F(t) = V(\wt K_t)$ is differentiable in $t\in [0,1]$, and
\[F'(t) = \frac d{dt} V(\wt K_t)  = \frac n{n-1}\int_{\sn} \log \lh( \frac{dS_{K_1}}{dS_{K_0}} \rh)   dV_{\wt K_t}. \]
Here the derivatives at endpoints of $[0,1]$ should be understood as one-sided derivatives.
\end{thm}

\begin{proof} Denote
\[\varphi = \frac{dS_{K_1}}{dS_{K_0}}.\]
By Lemma \ref{smalllem},  $\varphi\in L^1(\sn; S_{K_0})$ and $\varphi^{-1} \in L^1(\sn; S_{K_1})$. Then,
for $t=0,1$, the differential of $F(t)$ is already given by  Lemma \ref{var}.

Now, suppose $t\in (0,1)$. Choose $\varepsilon_0>0$ such that $[t-\varepsilon_0,t+\varepsilon_0]\subset (0,1).$ Define
\[\wt L_s = \wt K_{ t+s\varepsilon_0}, \quad  s\in[-1,1].\]
Then,  $S_{\wt L_{-1}}$,  $S_{\wt L_{0}}$, and $S_{\wt L_{1}}$ are absolutely continuous with respect to each other, and
\[\varphi^{-\varepsilon_0} = \frac{dS_{\wt L_{-1}}}{dS_{\wt L_{0}}}\in L^1(\sn;{S_{\wt L_{0}}}), \quad{\rm and}\quad
\varphi^{\varepsilon_0} = \frac{dS_{\wt L_{1}}}{dS_{\wt L_{0}}}\in L^1(\sn;{S_{\wt L_{0}}}).\]
By the two-sided variation formula (Lemma \ref{var2}), we have
\[ \frac d{dt}\Big|_{t=0} V(\wt L_s) \frac 1{n-1} \int_{\sn} \log \lh(\varphi^{\varepsilon_0}\rh)  h_{\wt L_{0}}dS_{\wt L_{0}}
=  \frac {n\varepsilon_0}{n-1}    \int_{\sn} \log  \varphi  ~ dV_{\wt L_{0}}.\]
Since $V(\wt L_s) = V(\wt K_{t+s\varepsilon_0})$, the above formula implies that $F(t)  = V(\wt K_t) $ is differentiable  and
\[F'(t) = \frac n{n-1}\int_{\sn} \log \varphi ~  dV_{\wt K_t}. \]
\end{proof}

Now we give the main result in this section. In the following two results, we split the statements of the conjectures into two parts: inequalities and equality cases. For example, (LBM) inequality means the inequality \eqref{lbmin} without the equality cases.

\begin{thm}[RTF Principle for the inequalities]\label{equi}  The following three items are equivalent.
\begin{itemize}
	\item[(i)] (LBM) inequality is true in $\rn$.
	\item[(ii)] (RLBM) inequality is true in $\rn$.
	\item[(iii)] For all    $K_0, K_1 \in \kn_e$, it follows that
	\be V(\wt K_t) \le V(  K_t).\ee
\end{itemize}	

\end{thm}

\begin{proof}
Denote $\mu = S_{K_0}+S_{K_1}$, $\varphi_0 = dS_{K_0}/d\mu$, $\varphi_1 = dS_{K_1}/d\mu$. Then
\[ dS_{K_0} = \varphi_0 d\mu, \quad   dS_{K_1} = \varphi_1 d\mu,\]
and
\[dS_{\wt K_t} = \varphi_0^{1-t}\varphi_1^td\mu = \lh(\frac{\varphi_1}{\varphi_0}\rh)^t dS_{K_0}.\]

\vskip 4pt
\noindent{\it Step 1: (i) implies (ii) and (iii).} Assume (LBM) inequality holds.

On one hand,  by H\"older's inequality,  we have
\be\label{e1}\begin{aligned}  \frac1n\int_{\sn} h_t dS_{\wt K_t} &=  \frac1n\int_{\sn} \lh({h_0}\varphi_0\rh)^{1-t} \lh({h_1}\varphi_1\rh)^{t} d\mu \\
	&\le \lh( \frac1n\int_{\sn} h_0 dS_{K_0}\rh)^{1-t} \cdot \lh( \frac1n\int_{\sn} h_1 dS_{K_1}\rh)^{ t}\\
	&=V(  K_0)^{1-t}V(  K_1)^t. \end{aligned} \ee
On the other hand, since $h_{t}\le h_{0}^{1-t}h_1^t$, by  Minkowski's inequality for the mixed volume, we have
\be \label{e2}\frac1n\int_{\sn} h_t dS_{\wt K_t} \ge V_1(\wt K_t, K_t) \ge V(\wt K_t)^{\frac{n-1}n}V( K_t)^{\frac{1}n}.\ee
Since we have assumed that (LBM) inequality holds, combining \eqref{e1} and \eqref{e2},  we obtain (ii):
\[ V(\wt K_t) \le V(  K_0)^{1-t}V(  K_1)^t.\]
It is clear that (iii)  follows immediately from (i) and (ii).
\vskip 4pt
\noindent{\it Step 2.}   We prove that  (iii) implies (ii).
By \eqref{e1} and \eqref{e2}, if (iii) holds, we have
\[V(\wt K_t) \le  V(  K_0)^{1-t}V(  K_1)^t.\]
\vskip 4pt
\noindent{\it Step 3.} We prove that  (ii) implies (i).

Suppose that $K_1$ is a minimizer of (\cite{BLYZ13} guarantee the existence of a minimizer)
\[\inf \Big\{ \frac{1}{V({K_0})}\int_{\sn}\log h_L dV_{K_0} - \frac 1n\log V(L)  : L\in\kn_e  \Big\}\]
satisfying $V(K_1) = V(K_0)$. We infer from \cite{BLYZ13} (via the variation approach) that
\[V_{K_1} = V_{K_0}.\]
Then, we have $\varphi =dS_{K_1}/dS_{K_0} = h_0/h_1$ is continuous and positive on $\sn$. This allows us to use Lemma \ref{var}.
Thus, the following function  $F(t)$
\[ F(t) := V(K_0)^{1-t}V(K_1)^t - V(\wt K_t)  \ge 0, \quad \forall t\in[0,1],\]
is differentiable at $0$.
Since  $F(0) = 0$, the right differential $F_r'(0)$ satisfies
\[ 0 \le F_r'(0) =  - \frac1{n-1}\int_{\sn} \log\varphi dV_{K_0}.\]
Since $\varphi={dS_{K_1}}/{dS_{K_0}} = h_0/h_1$, we have
\[\int_{\sn} \log {h_1} dV_{K_0} \ge \int_{\sn} \log {h_0} dV_{K_0}.\]
Recalling that $h_1$ is a minimizer, we must have that $h_0$ itself is always a minimizer of the problem
\[\inf \Big\{ \frac{1}{V({K_0})}\int_{\sn}\log h_L dV_{K_0} - \frac 1n\log V(L)  : L\in\kn_e  \Big\}.\]
In other words, for all ${K_0}\in \tk_{e,+}^{n,\infty}$ and for all $K_0\in \kn_e$,  there is the (LM) inequality
\[\int_{\sn}\log \frac{h_{K_1}}{h_{K_0}} dV_{K_0} \ge \frac{V(K_0)}n\log \frac{V(K_1)}{V(K_0)}. \]

By  \cite[Lemma 3.2]{BLYZ12}, the (LBM) inequality holds for all $K_0,K_1\in\kn_e$ (without the equality condition).
\end{proof}

A careful examination of the proof of Theorem \ref{equi} leads to the following:
\begin{cor}[RTF Principle for the equality condition]\label{equi-eq} If (LBM) inequality holds for all bodies in $\kn_e$, then the following three items are equivalent.
\begin{itemize}
	\item[(i)] Equality holds in the (LBM) inequality if and only if $K_0$ and $K_1$ have dilated summands.
	\item[(ii)] Equality holds in the (RLBM) inequality if and only if $K_0$ and $K_1$ have dilated summands.
	\item[(iii)] Equality holds in
	\[ V(\wt K_t) \le V(  K_t)\]
 if and only if $K_0$ and $K_1$ have dilated summands.
\end{itemize}	
\end{cor}

\section{A very simple new proof of (LBM) in dimension 2}

Following from the ``reverse-to-forward principle" in Section 2, we give a very simple new proof of  (LBM) in dimension 2.

\begin{thm}   (RLBM) conjecture and (LBM) conjecture are true in dimension 2.
\end{thm}

\begin{proof}
Let $K_0, K_1 \in \tk^2_e$. Denote $\mu = S_{K_0}+S_{K_1}$, and denote
\[\varphi_0  = \frac{dS_{K_0}}{d\mu}, \quad \varphi_1  = \frac{dS_{K_1}}{d\mu}.\]

Denote $u^\perp$ to be the unit vector  obtained by ratating $u$ by an angle of $\pi/2$ clockwise. Notice that for an origin-symmetric planar convex body, $h_K(u^\perp)$ is equal to the length of the projection of $K$ onto $u^\perp$. Therefore,
by H\"older's inequality,
\be \label{e2*}
\begin{aligned}
	& 4h_{\wt K_t}(u^\perp)\\
	 =& \int_{S^1}|u\cdot v| dS_{\wt K_t}(v) \\
	 =&  \int_{S^1}|u\cdot v|  \varphi_0(v)^{1-t}\varphi_1(v)^td{\mu}(v) \\
	 \le& \lh(\int_{S^1}|u\cdot v| dS_{K_0}(v)\rh)^{1-t}\lh(\int_{S^1}|u\cdot v|  dS_{K_1}(v)\rh)^{t}\\
	 =& 4h_{\wt K_0}(u^\perp)^{1-t}h_{\wt K_1}(u^\perp)^{t}.
\end{aligned}\ee
Since $u$ is arbitrary, we have $\wt K_t \subset K_t$, and hence
\[V(\wt K_t)\le V( K_t).\]
This proves (iii) of Theorem \ref{equi}, and hence proves (RLBM) inequality and (LBM)  inequality  in dimension 2.

Equality holds in (iii) of Theorem \ref{equi} if and  only if equality holds in \eqref{e2*} for each $u$ satisfying $u^\perp\in \Omega = {\rm supp} \mu$. Then, for each $u$ satisfying $u^\perp\in \Omega$,
$\varphi_1/\varphi_0$ is a constant (depending on $u$) w.r.t. the measure $|u\cdot v|d\mu(v)$. There are only two cases: Either $\Omega $ has only two pairs of opposite vectors, or $\varphi(v)$ is a constant a.e. on $\Omega$.
As a result, either $K_0$ and $K_1$ are parallelotope with parallel sides, or $K_0$ and $K_1$ are dilates. In other words, they   have dilated direct summands.

This proves (iii) in Corollary \ref{equi-eq}, and hence proves  the equality cases of (RLBM) conjecture and (LBM) conjecture in dimension 2.
\end{proof}

\section{  (LM) in the case that the `base body' is a zonoid}

We will prove (LM) in the case that $K$ is a zonoid,
\[ \int_{\sn}\log\frac{h_L}{h_K}dV_K \ge \frac{V(K)}n\log\frac{V(L)}{V(K)},\]
and we will characterize the full equality conditions. The inequality was proved by  van Handel \cite{H2}, via the estimate of spectral-gap of the Hilbert-Brunn-Minkowski operator defined by Kolesnikov-Milman \cite{KM}. The philosophical background of his proof was the theories developed by  Kolesnikov-Milman \cite{KM} (local version) and Chen-Huang-Li-Liu \cite{CHLL} (local-to-global principle).
Our proof is based on the ``reverse-to-forward" principle, without the use of previous theories.
The characterization of full equality conditions turns to be new.


\subsection{Reverse-to-forward principle for (LM)}

\begin{lem}[Reverse-to-forward principle for (LM)]\label{rtf-lm}
	Let $K\in \kn_e$.	Then, in the subsequent context,  (i) entails   (ii).
	\begin{itemize}
		\item[(i)] (Weak Reverse-log-Minkowski inequality)  For all $L\in \kn_e$ such that  $S_K$ and $S_L$ are absolutely continuous with respect to each other, it holds that
		\be\label{eq-weaklm} \frac1{n-1}\int_{\sn}\log \lh(\frac{dS_L}{dS_K}\rh)   dV_{K} \le  \int_{\sn}\log \frac{h_{L} }{h_{K} }  dV_{K}.\ee	
		
		\item[(ii)] (Log-Minkowski inequality) For all $L\in \kn_e$, it holds that
		\be \label{lm} \int_{\sn}\log \frac{h_{L} }{h_{K} }  dV_{K} \ge \frac{V(K)}n\log\frac{V(L)}{V(K)}. \ee
		
	\end{itemize}
\end{lem}

Notice that under the smoothness assumption that $\partial K$ is $C^2$, Ma-Zeng-Wang \cite{MZW} also established the equivalence result. 

\begin{proof}
	Assume (i) is true.
	Suppose $M\in \kn_e$ is a minimizer (its existence is provided by \cite{BLYZ13}) of
	\be\label{miniprob}\inf \Big\{ \int_{\sn}\log h_L dV_{K} : L\in\kn_e  ~\text{and}~V(L) = V(K)  \Big\}.\ee
	We infer from \cite{BLYZ13} (via a variation approach) that
	\[V_M = V_{K},\]
	which implies that $\varphi =dS_M/dS_{K} = h_{K}/h_M$, which is positive and continuous. Then, by \eqref{eq-weaklm} and taking $L=M$, we have
	\[\frac1{n-1}\int_{\sn}\log  \frac{h_K}{h_M}    dV_{K} \le  \int_{\sn}\log \frac{h_{M} }{h_{K} }  dV_{K}.\]
	As a result,
	\[\int_{\sn}\log  {h_{M} }   dV_{K}\ge \int_{\sn}\log  h_{K}    dV_{K}.\]
	But $M$ is a minimizer  of \eqref{miniprob}, $K$ must also be a minimizer. This proves the inequality
	\[  \int_{\sn}\log \frac{h_L}{h_K} dV_{K} \ge 0,  \]
	for all $L\in\kn_e $ with $V(L) = V(K)$.  Then, proposition (ii) follows immediately from   homogeneity.
\end{proof}

\begin{lem}[The reverse-log-Minkowski inequality for projections] \label{thm-weakrlm}
	Let $K, L\in \kn_e$ satisfy that $S_K$ and $S_L$ are absolutely continuous with respect to each other. Then, for each $u\in \sn,$	
		\be\label{eq-rlm-proj-nd}\int_{\sn} \log\varphi(v) |u\cdot v| dS_{K_0}(v) \le2{\vol_{n-1}(P_{u^\perp} K_0)} \log\lh(\frac{\vol_{n-1}(P_{u^\perp} K_1)}{\vol_{n-1}(P_{u^\perp} K_0)}\rh),\ee
	with equality if and only if there is a constant $c(u)$ and an $S_K$-full measure set $\Omega$, such that
	\[ \frac{dS_L}{dS_K}(v) = c(u), \quad \forall v\in \Omega  \setminus u^\perp.\]

\end{lem}

\begin{proof}
	Denote $\varphi =  {dS_L}/{dS_K}$. By Lemma \ref{smalllem}, $\varphi$ is positive and finite $S_K$-a.e.
	
	Let $K_0=K$ and $K_1=L$ , and consider their log-Blaschke combination $\wt K_t =(1-t)\cdot \wt K_{0} \#_0  t\cdot \wt  K_{1}$. For each $u\in \sn$, by H\"older's inequality, we have \footnote{Here we use $|P_{u^\perp} K|$ to denote the $(n-1)$-volume $\vol_{n-1}(P_{u^\perp} K)$. We are doing this to compare which one looks better.}
	\be\label{3dproj}
	\begin{aligned}
		|P_{u^\perp} \wt K_t| =& \frac12\int_{\sn} |u\cdot v| dS_{\wt K_t}(v) \\
		=&\frac12\int_{\sn}  |u\cdot v| \varphi(v)^t dS_{ K_0}(v)\\
		\le & \lh(\frac12\int_{\sn} |u\cdot v| dS_{K_0}(v)\rh)^{1-t}\lh(\frac12\int_{\sn} |u\cdot v|\varphi(v)  dS_{K_0}(v)\rh)^{t}\\
		=& |P_{u^\perp}  K_0|^{1-t}|P_{u^\perp} K_1|^t.
	\end{aligned}  \ee
	Since for any $0\le t_0<t_1\le 1$, there is
	\[\wt K_{(1-s)t_0 + st_1} = (1-s)\cdot \wt K_{t_0} \#_0  s\cdot \wt  K_{t_1},\]
	the inequality \eqref{3dproj} also holds for  $\wt K_{t_0}$ and $\wt K_{t_1}$, namely,
	\be\label{eq-rlbm-proj}|P_{u^\perp} \wt K_{(1-s)t_0 + st_1}| \le |P_{u^\perp} \wt K_{t_0}|^{1-s}|P_{u^\perp} \wt K_{t_1}|^s, \quad \forall s\in (0,1).\ee
	Denote $\mu_u$ to be the measure
	\[d\mu_u(v) = |u\cdot v|dS_{K}(v).\]
	Since $\varphi$ is positive $S_K$-a.e.,
	equality holds in \eqref{eq-rlbm-proj} if and only if
	$\varphi(v)$ is a constant (depending on $u$)   $\mu_u$-a.e.
	Then, \eqref{eq-rlbm-proj} implies that
	\[F(t) = \log\lh(\vol_{n-1}(P_{u^\perp}\wt K_t)\rh)\]
	is convex on $[0,1]$, and is strictly convex if $\varphi(v)$ is not a constant $\mu_u$-a.e. It follows that $F'(0) \le F(1) - F(0)$, which is equivalent to
	\be\label{eq-rlm-proj}\int_{\sn} \log\varphi(v) |u\cdot v| dS_{K_0}(v) \le2{\vol_{n-1}(P_{u^\perp} K_0)} \log\lh(\frac{\vol_{n-1}(P_{u^\perp} K_1)}{\vol_{n-1}(P_{u^\perp} K_0)}\rh).\ee
	If equality holds in \eqref{eq-rlm-proj}, then $\varphi$ cannot be strictly convex, and hence
	there is a constant $c(u)$ such that
	$\varphi (v) = c(u)$,    $\mu_u$-{\rm a.e.},
	Therefore, there is an $S_K$-full measure set $\Omega$, such that
		\[ \varphi(v) = c(u), \quad \forall v\in \Omega  \setminus u^\perp.\]
	This shows the necessary condition for the equality to hold, and the sufficient condition is evident.
\end{proof}

\subsection{A Lemma  to characterize the equality conditions}

The following lemma of  relations is the crucial part of characterizing the equality conditions of the weak reverse-log-Minkowski inequality and (LM).

\begin{lem}[Lemma of relations]\label{lem-relation}
	Let $\Omega, \omega \subset \sn$ satisfy ${\rm span}(\Omega) = {\rm span} (\omega) = \rn$. Define the relations $\bowtie$ and $\sim$ as follows.
	For $v\in \Omega$ and $u\in \omega$, we say $v \bowtie u$ (or $u\bowtie v$), if $v\notin u^\perp$.
	For $v_1, v_2\in \Omega$, we say $v_1\sim v_2$, if there is $u\in \omega$ such that $v_1\bowtie u$ and $v_2\bowtie u$;
	For $u_1, u_2\in \omega,$ we say  $u_1\sim u_2$, if there is $v\in \Omega$ such that $u_1\bowtie v$ and $u_2\bowtie v$.
	
	Define the equivalence $\simeq$ as follows. For $v_1,v_2\in \Omega$, we say $v_1\simeq v_2$, if there are if there are $v_{i_1},...,v_{i_j}\in \Omega$, such that
	\[v_1 \sim v_{i_1} \sim \cdots \sim v_{i_j} \sim v_2.\]
	For $u_1,u_2\in \omega$, we say $u_1\simeq u_2$, if there are if there are $u_{i_1},...,u_{i_k}\in \omega$, such that
	\[u_1 \sim u_{i_1} \sim \cdots \sim u_{i_k} \sim u_2.\]
	
	Suppose that the quotient spaces $\lh(\Omega/\simeq\rh) = \{[v_i]: i\in I\}$ and $\lh(\omega/\simeq\rh) = \{[u_j]: i\in J\}$	with index sets $I$ and $J$. Denote the linear spans of the representatives by
	\[V_i = {\rm span} \{v\in \Omega : v\in [v_i] \}, \quad i\in I\]
	and
	\[U_j = {\rm span} \{u\in \omega : u\in [u_j] \}, \quad j\in J.\]
	
	Then, $I = J$ is a finite set, and we denote it by $I= J=\{1,...,i_0\}$. Moreover, $\rn$ can be decomposed into the direct sum of $V_i$'s as well as the direct sum of $U_j$'s, that is
	\[\rn = V_1\oplus\cdots\oplus V_{i_0} = U_1\oplus\cdots\oplus U_{i_0}  .\]
\end{lem}

\begin{proof}
	It is clear that   $\simeq$ is an equivalence relation. By the assumption  ${\rm span}(\Omega) = {\rm span} (\omega) = \rn$, for each $v\in \Omega$, we can find a $u\in\omega$ so that
	$v\bowtie u$.
	
	With out loss of generality, assume $v_1\bowtie u_1$, ${\rm span} [v_1] = V_1$, and ${\rm span} [u_1] = U_1$. Denote
	\[\wt V_2 = {\rm span}\{v \in [v_i]: \quad \forall i\in I ~\text{with}~  i\ne 1\}  \]
	and
	\[\wt U_2 = {\rm span}\{u \in [U_j]: \quad \forall j\in J ~\text{with}~  j\ne 1\}.  \]
	By the definition of $\bowtie$, for any $v\in[v_1]$ and $u\in [u_j]$ with $j\ne 1$, we must have
	$v\perp u$. In fact, there are $v_{i_1},...,v_{i_k}\in\Omega$ and $u_{i_1},...,u_{i_k},u_{i_{k+1}}\in\omega$ such that
	\[v\bowtie u_{i_1} \bowtie v_{i_1}\bowtie \cdots u_{i_k}
	\bowtie v_{i_k}  \bowtie u_{i_{k+1}}\bowtie v_1 \bowtie u_1, \]
	If $u\bowtie v$, we will have $u\sim u_1$, which contradits the assumption $u\notin [u_1]$.
	
	As a result, we have $V_1\perp \wt U_2$, and  similarly $U_1\perp \wt V_2$. This, together with the fact that
	\[\rn = V_1 + \wt V_2 =  U_1 + \wt U_2,\]
	implies
	\[ n\le \dim V_1 + \dim \wt V_2 \le (n-  \dim \wt U_2) + (n-  \dim U_1) \le n.\]
	Then, it must be the case that
	\[V_1 =  \wt U_2^\perp, \quad U_1 =  \wt V_2^\perp,\]
	and $\rn$ is the direct decomposition of $V_1, \wt V_2$ as well as $U_1, \wt U_2$:
	\[\rn =  V_1 \oplus \wt V_2 =  U_1 \oplus \wt U_2.\]
	
	Observe the following facts:
	\begin{itemize}
		\item 	$\Omega =  (\Omega \cap   V_1)\cup (\Omega \cap \wt V_2),$  {\rm and} $(\omega \cap   U_1)\cup (\omega \cap \wt U_2)$;
		\item $V_i\subset \wt V_2$ and $U_i\subset \wt U_2$ for all $i\ne1$;
		\item For $v\in\Omega\cap  \wt V_2$ and $v'\in \Omega$, if $v\simeq v'$, then $v'$ must belong to $\Omega\cap \wt V_2$.
	\end{itemize}
	Then, the new quotient spaces $(\Omega\cap\wt V_2 )/\simeq$ and $(\omega\cap\wt U_2) /\simeq$ are exactly
	$\{[v_i] : i\in I\setminus\{1\}\}$ and $\{[u_j] : j\in J\setminus\{1\}\}$.
	
	Thus, we can do similar things in $\Omega\cap\wt V_2$ and $\omega\cap\wt U_2$,
	and after finite steps, we will  complete the proof.
\end{proof}


\subsection{The log-Minkowski inequality for that the `base' convex body is a zonoid}

We say $K$ is a zonoid in $\rn$ with {\it generating measure} $\rho$, if
\[h_K(v) = \int_{\sn}|u\cdot v| d\rho(v), \]
where $\rho$ is a nonnegative finite Borel measure.

Suppose $Z$ is a zonoid with generating measure $\rho$.
Since the mixed surface area measure $S(K_1,...,K_{n-1},\cdot)$ is linear in each argument $K_i$, it follows from an approximation process that
\be\label{zonoidmv}
\frac1n\int_{\sn}\int_{\sn}f(w)dS(K,..., K,[-u,u],w)d\rho(u) =	  \frac1{n} \int_{\sn}f(v)dS(K,..., K,Z,v),
\ee
for each $f\in C(\sn)$. This formula can also be viewed as a corollary of  \cite[Theorem 5.3.2]{Sch}.

Since the mixed volume $V(K_1,...,K_{n-1},K_n)$ is linear in each argument,
\[v^{(n-1)} (P_{u^\perp}K,...,P_{u^\perp}K,P_{u^\perp}L) = \int_{\sn} |u\cdot v| d S(K,...,K,L,v) = \int_{\sn} h_L d S(K,...,K,[-u,u],v).\]
Here $v^{(n-1)}$ denotes the $(n-1)$-dimensional mixed volume.
Since $L$ is arbitrary, and since the linear span of the set of support functions is dense in $C^2(\sn)$, we
have
\[ \int_{\sn\cap u^\perp} f(v) d S_{P_{u^\perp}K}^{(n-1)}(v) = \int_{\sn} f(v) d S(K,...,K,[-u,u],v).\]
Here $S^{(n-1)}$ means the $(n-1)$-dimensional mixed surface area measure in $u^\perp$.

Therefore,
for each $u\in \sn$, there is
\be\label{sam-proj}
S_{P_{u^\perp}K}^{(n-1)}(\omega) 	=  S(K,...,K,[-u,u], \omega), \quad {\rm for~Borel~} \omega\subset \sn. \ee
This can also be seen in van Handel's papers [H1,H2].

	\begin{thm}\label{ndlm}
		Let $K$ be a zonoid in $\rn$ with generating measure $\rho$, and let $L\in \tk_e^n$. Then,
	\be\label{lm-zonoid-p} \int_{\sn}\log\frac{h_L}{h_K}dV_K \ge \frac{V(K)}n\log\frac{V(L)}{V(K)}, \ee
		with equality if and only if $K$ and $L$ have dilated direct summands. 	
	\end{thm}

	\begin{proof}
		The proof is split into two parts.
		
		In the first part,  , we will prove the (LM) inequality, by induction.
			 Then,	we characterize the equality conditions in the second part.
		
		\vskip 4pt
	
\noindent{\it Part 1: The inequality \eqref{lm-zonoid-p}.}   By the reverse-to-forward principle 2 (Lemma \ref{rtf-lm}), it suffices to prove \eqref{eq-weaklm}, under the assumption that $S_K$ and $S_L$ are absolutely continuous with 	respect to each other. We will prove it by induction.
\begin{itemize}
	\item {\it Step 1:  $n=3$.}
	
		Denote	$\varphi =  {dS_L}/{dS_K}.$
		By \eqref{eq-rlm-proj-nd}, we have
		\be\label{eq-rlm-proj}\int_{S^2} \log\varphi(v) |u\cdot v| dS_{K}(v) \le2{\vol_2(P_{u^\perp} K)} \log\lh(\frac{\vol_2(P_{u^\perp} L)}{\vol_2(P_{u^\perp} K)}\rh).\ee	
		Since  for $v\in S^2\cap u^\perp$,
		\[h_{P_{u^\perp}  K}(v) = h_{K}(v), \quad  h_{P_{u^\perp}  L}(v) = h_{L}(v),\]
		applying  2-dimensional log-Minkowski inequality in $\tr^3\cap u^\perp$, and by \eqref{sam-proj}, we obtain
		\[
		\begin{aligned}
			& {\vol_2(P_{u^\perp} K)} \log\lh(\frac{\vol_2(P_{u^\perp} L)}{\vol_2(P_{u^\perp} K)}\rh)\\
			\le& \int_{S^2\cap u^\perp} \log \frac{h_{L}(v)}{h_{K}(v)} h_{K}(v)dS_{P_{u^\perp}K}(v)\\
			= & \int_{S^2 } \log \frac{h_{L}(v)}{h_{K}(v)} h_{K}(v)dS(  K, [-u,u], v).
		\end{aligned}
		\]
		Integrating this inequality with respect to the generating measure $\rho$,  by \eqref{zonoidmv} and \eqref{eq-rlm-proj},
		we get
		\be\label{eq-weaklm2d}\int_{S^2}\log \varphi  h_{K} dS_{K} \le 2 \int_{S^3}\log \frac{h_{L} }{h_{K} } h_{K}  dS_{K}.\ee
	 By  Lemma \ref{rtf-lm},  \eqref{lm-zonoid-p} follows immediately for all $L\in \tk^3_e$.
		
\item	{\it Step 2: From dimension $(n-1)$ to dimension $n$.}

		Assume that   (LM)   \eqref{lm-zonoid-p} is true in dimension $n-1$ with $n\ge 4$, in the case that the `base' convex body $K$ is a zonoid.

		Since $P_{u^\perp}K$ is   an $(n-1)$-dimensional zonoid in $\rn\cap u^\perp$, by induction, and by \eqref{sam-proj}, we have
			\[
		\begin{aligned}
			& {\vol_{n-1}(P_{u^\perp} K)} \log\lh(\frac{\vol_{n-1}(P_{u^\perp} M)}{\vol_{n-1}(P_{u^\perp} K)}\rh)\\
			\le& \int_{\sn\cap u^\perp} \log \frac{h_{M}(v)}{h_{K}(v)} h_{K}(v)dS_{P_{u^\perp}K}(v)\\
			= & \frac{n-1}2\int_{\sn } \log \frac{h_{M}(v)}{h_{K}(v)} h_{K}(v)dS(  K, [-u,u], v).
		\end{aligned}
		\]
It follows from this and \eqref{eq-rlm-proj-nd} that
		\be\label{weak-rlm-nd2} \int_{\sn } \log \varphi(v)|u\cdot v| dS_K(v)  \le (n-1)\int_{\sn } \log \frac{h_{L}(v)}{h_{K}(v)} h_{K}(v)dS(  K, [-u,u], v), \ee
		for each $u\in \sn$.  Integrating this inequality with respect to $\rho$, and using \eqref{zonoidmv}, we obtain
		\be\label{weak-rlm-nd2} \frac1{n-1} \int_{\sn } \log \varphi(v)  dV_K(v)  \le \int_{\sn } \log \frac{h_{L}(v)}{h_{K}(v)}  dV_K(v). \ee
	 By  Lemma \ref{rtf-lm},  \eqref{lm-zonoid-p} follows immediately for all $L\in \kn_e$.

\end{itemize}	
		
	\noindent{\it Part 2: Characterization of   equality conditions.}
	\vskip 4pt
	Since the sufficient part is clear, we only need to prove the necessary part.
	
	Suppose $L\in \kn_e$ achieves equality in \eqref{lm-zonoid-p}. Then, a dilation of $L$, say $aL$, will satisfy $V_{aL}=V_K$. Hence,
	equality also holds in 	\eqref{weak-rlm-nd2} for $K$ and $L$. As a result, there is a $\rho$-full measure set $\omega$, such that
	equality holds in \eqref{eq-rlm-proj-nd} for all $u\in\omega$.
	By Lemma \ref{thm-weakrlm}, there is an $S_K$-full measure set $\Omega$ (may depend  on $u$ presently), such that
	\[ \varphi(v) = \frac{dS_L}{dS_K}(v) = c(u), \quad \forall v\in \Omega  \setminus u^\perp,\]
	where $c(u)$ is a constant depending on $u\in \omega$. Since $\varphi= a^{-n}h_K/h_L$ is continuous, the set $\Omega$ can be chosen to be ${\rm supp} (S_K)$.
	
Now we recall the notations $\bowtie$, $\sim$ and $\simeq$ in Lemma \ref{lem-relation}.
By   Lemma \ref{lem-relation},
$\rn$ is decomposed into the direct sum of linear spans of the equivalence classes, $V_1,...,V_{i_0}$.
By the construction of $V_i$, there is
\[\Omega = \bigcup_{i=1}^{i_0} (V_i\cap \Omega).\]
Notice the facts that
\begin{itemize}
	\item If $v_1\simeq v_2$ in $\Omega$, then $\varphi(v_1)=\varphi(v_2)$;
		\item If $u_1\simeq u_2$ in $\omega$, then $c(u_1)=c(u_2)$.
\end{itemize}
Therefore, there are positvie numbers $c_1,...,c_{i_0}$, such that
\[ \varphi(v) = h_K(v)/h_L(v) = c_i, \quad \forall v\in \Omega\cap V_i.\]
Equivalently to say, $K$ and $L$ have dilated direct summands.
	\end{proof}

\noindent{\bf Acknowledgement.} The author thank Ramon van Handel for his very valuable discussions on earlier versions to improve the presentation.

\bibliographystyle{abbrv}

\end{document}